\documentclass[11pt,reqno]{amsart}
\usepackage{amsmath,amssymb,latexsym,soul,cite,mathrsfs}

\usepackage{color,enumitem,graphicx}
\usepackage[colorlinks=true,urlcolor=blue,
citecolor=red,linkcolor=blue,linktocpage,pdfpagelabels,
bookmarksnumbered,bookmarksopen]{hyperref}
\usepackage[english]{babel}
\usepackage{enumitem}

\usepackage[left=2.7cm,right=2.7cm,top=2.9cm,bottom=2.9cm]{geometry}

\usepackage[hyperpageref]{backref}

\numberwithin{equation}{section}

\newtheorem{theorem}{Theorem}[section]
\newtheorem{lemma}[theorem]{Lemma}
\newtheorem{corollary}[theorem]{Corollary}
\newtheorem{proposition}[theorem]{Proposition}

\newtheorem{remark}[theorem]{Remark}

\renewcommand{\epsilon}{\varepsilon}

\renewcommand{\theta}{{\vartheta}}

\renewcommand{\rightarrow}{\to}

\title[Ground states for a coupled system of Kirchhoff-Schr\"odinger equations]{Positive ground states for a subcritical and critical coupled system involving Kirchhoff-Schr\"odinger equations }  

\author[J.C. de Albuquerque]{Jos\'{e} Carlos de Albuquerque}
\author[J.M.\ do \'O]{Jo\~ao Marcos do \'O}
\author[Giovany M. Figueiredo]{Giovany M. Figueiredo}

\address[J.C. de~Albuquerque]{Instituto de Matem\'{a}tica e Estat\'{i}stica,
	UFG - Universidade Federal de Goi\'{a}s,
	\newline\indent
	74001-970, Goi\'{a}s-GO, Brazil}
\email{\href{mailto:joserre@gmail.com}{joserre@gmail.com}}

\address[J.M. do \'O]{Departamento de Matem\'{a}tica, UnB - Universidade de Bras\'ilia,
	\newline\indent
	70297-400, Bras\'{i}lia-DF, Brazil}
\email{\href{mailto:jmbo@pq.cnpq.br}{jmbo@pq.cnpq.br}}

\address[G. M. Figueiredo]{Departamento de Matem\'{a}tica, UnB - Universidade de Bras\'ilia,
\newline\indent
70297-400, Bras\'{i}lia-DF, Brazil}
\email{\href{mailto:giovany@unb.br}{giovany@unb.br}}

\thanks{Research supported in part by INCTmat/MCT/Brazil, CNPq and CAPES/Brazil}
 

\begin{document}
	

\begin{abstract}
	In this paper we prove the existence of positive ground state solution for a class of linearly coupled systems involving Kirchhoff-Schr\"{o}dinger equations. We study the subcritical and critical case. Our approach is variational and based on minimization technique over the Nehari manifold. We also obtain a nonexistence result using a Pohozaev identity type.
\end{abstract}
	
\maketitle


\section{Introduction}

In this article we study the following class of nonlocal linearly coupled systems 
\begin{equation}\label{paper1j0}
\left\{
\begin{array}{lr}
\left(a_{1}+\alpha^{\prime}(\|u\|_{E_{1}}^{2}) \right)(-\Delta u+V_{1}(x)u) =\mu|u|^{p-2}u+\lambda(x)v, & x\in\mathbb{R}^{3},\\
\left(a_{2}+\beta^{\prime}(\|v\|_{E_{2}}^{2}) \right)(-\Delta v+V_{2}(x)v) =|v|^{q-2}v+\lambda(x)u,    & x\in\mathbb{R}^{3},
\end{array} \tag{$S_{\mu} $}
\right.
\end{equation} 
where $a_{1},a_{2}>0$, $\alpha,\beta\in C^{2}(\mathbb{R}_{+},\mathbb{R}_{+})$ and for each $i=1,2$ we consider the following weighted Sobolev space and norm
$$
E_{i  }:=\left\{w\in H^{1}(\mathbb{R}^{3}):\int_{\mathbb{R}^{3}}V_{i  }(x)w^{2}\;\mathrm{d} x<\infty\right\}, \quad \|w\|_{E_{i}}^{2}=\int_{\mathbb{R}^{3}}|\nabla w|^{2}\;\mathrm{d} x+\int_{\mathbb{R}^{3}}V_{i  }(x)w^{2}\;\mathrm{d} x.
$$
Moreover, we assume that the coupling term $\lambda(x)$ is related with the potentials by $|\lambda (x)|\leq\delta\sqrt{V_{1  }(x)V_{2  }(x)}$, for some suitable $\delta>0$. Our main contribution here is to prove the existence of positive ground states for the subcritical case, that is, when $4<p\leq q<2^{*}=6$ and for the critical case when $4<p<q=6$. In the critical case, the existence of ground states will be related with the parameter $\mu$ introduced in the first equation. In fact, we obtain the existence result when $\mu>0$ is large enough. For the critical case when $p=q=6$, we make use of the Pohozaev identity type to prove that System~\eqref{paper1j0} does not admit positive solution.

 In order to motivate our results we begin by
giving a brief survey on Kirchhoff problems. First, we mention that System~\eqref{paper1j0} is called \textit{nonlocal} due the dependence of the norms $\|.\|_{E_{1}}$ and $\|.\|_{E_{2}}$. This type of equations take care of the behavior of the solution in the whole space, which implies that the equations in \eqref{paper1j0} are no longer a pointwise identity. We point out that nonlocal problems have been applied in many different contexts, for instance, biological systems, where can be used to describe the growth and movement of a particular species. Moreover, we cite also conservation laws, applications on population density, etc. The Kirchhoff-Schr\"{o}dinger equations introduced in System~\eqref{paper1j0} are also motivated by some physical models. The first study in this direction was proposed by Kirchhoff \cite{kirchhoff} in the study of the following hyperbolic equation
\begin{equation}\label{jgj2}
\rho\frac{\partial^{2}u}{\partial t^{2}}-\left(\frac{P_{0}}{h}+\frac{E}{2L}\int_{0}^{L}\left|\frac{\partial u}{\partial x}\right|\; \mathrm{d}x\right)\frac{\partial^{2}u}{\partial x^{2}}=0.
\end{equation}
Equation \eqref{jgj2} is a generalization of the classical d'Alembert's wave equation, by considering the vibrations of the strings. 

When $\lambda\equiv0$, the uncoupled Kirchhoff-Schr\"{o}dinger equation from System~\eqref{paper1j0} is related with the following stationary type equation
\begin{equation}\label{jgj3}
\left\{
\begin{array}{cl}
u_{tt}-\displaystyle M\left(\int_{\Omega}|\nabla u|^{2}\;\mathrm{d}x\right)\Delta u=f(x,u), & \mbox{in} \; \Omega\times (0,T),\\
u=0, & \mbox{on} \; \partial\Omega\times(0,T),\\
u(x,0)=u_{0}(x), & \mbox{in} \; \Omega,\\
u_{t}(x,0)=u_{1}(x), & \mbox{in} \; \Omega, 
\end{array}
\right.
\end{equation}  
where $\Omega\subset\mathbb{R}^{N}$ is a bounded domain. In 1878, J.L. Lions \cite{lions} introduced a functional analysis approach to study problem \eqref{jgj3}. Motivated by the physical interest and impulsed by \cite{lions}, Kirchhoff problems has been extensively studied by many authors in the last years. Concerning the scalar case related with \eqref{jgj3}, there are several works with respect to the following Kirchhoff--Schr\"{o}dinger equation
\begin{equation}\label{scalar}
\left\{
\begin{array}{cl}
\displaystyle-\left(a+b\int_{\Omega}|\nabla u|^{2}\;\mathrm{d}x\right)\Delta u=f(x,u), & \mbox{in} \; \Omega,\\
u=0, & \mbox{on} \; \partial\Omega,
\end{array}
\right.
\end{equation}
where $a,b>0$ and $\Omega\subset\mathbb{R}^{3}$ is a smooth bounded domain. For existence and multiplicity of solutions for related problems to \eqref{scalar}, we refer the readers to \cite{alves,chen,he,tang} and references therein. 

Here we are concerned to study the existence of ground states for a class of nonlocal linearly coupled systems defined in the whole space $\mathbb{R}^{3}$. We are motivated by recent works which obtain existence of solutions for nonlocal systems by using a variational approach. In this direction, D. L\"{u} and J. Xiao \cite{lu} studied the following class of coupled systems involving Kirchhoff equations
\begin{equation}\label{jgj4}
\left\{
\begin{array}{ll}
\displaystyle-\left(a+b\int_{\mathbb{R}^{3}}|\nabla u|^{2}\;\mathrm{d}x\right)\Delta u+\lambda V(x)u=\frac{2\alpha}{\alpha+\beta}|u|^{\alpha-2}u|v|^{\beta}, & x\in\mathbb{R}^{3},\\
\displaystyle -\left(a+b\int_{\mathbb{R}^{3}}|\nabla v|^{2}\;\mathrm{d}x\right)\Delta v+\lambda W(x)v=\frac{2\beta}{\alpha+\beta}|u|^{\alpha}|v|^{\beta-2}v, & x\in\mathbb{R}^{3},\\
u(x)\rightarrow0, \ v(x)\rightarrow0, \ as \ |x|\rightarrow\infty,
\end{array}
\right.
\end{equation} 
where $\alpha,\beta>2$ satisfying $\alpha+\beta<2^{*}=6$. The authors obtained existence and multiplicity of solutions when the parameter $\lambda>0$ is large. For more related results to the class of coupled systems \eqref{jgj4} we refer the readers to \cite{xiao,shi}.


Motivated by the above discussion, our purpose is to study the class of coupled systems \eqref{paper1j0}, by considering a more general class of Kirchhoff-Schr\"{o}dinger equations. The class of systems \eqref{paper1j0} imposes several difficulties. The first one is that the nonlocal terms here are introduced by functions $\alpha,\beta\in C^{2}(\mathbb{R}_{+},\mathbb{R}_{+})$ which generalize the standard Kirchhoff-Schr\"{o}dinger equations (see Remark~\ref{r1}). Moreover, we deal with the ``lack of compactness" due the fact that the problem is defined in the whole space $\mathbb{R}^{3}$. Another obstacle is the fact that System~\eqref{paper1j0} involves strongly coupled Kirchhoff-Schr\"{o}dinger equations because the linear terms in the right hand side. The work is divided into three parts: The first one consists to study System~\eqref{paper1j0} in the subcritical case, that is, when $4<p\leq q<2^{*}=6$. By using a minimization method over the Nehari manifold we obtain the existence of at least one positive ground state solution for System~\eqref{paper1j0}, for any parameter $\mu>0$. After that, we study the critical case, when $4<p<q=6$. In this case, we use the parameter $\mu>0$ to control the range of the ground state energy level associated to System~\eqref{paper1j0}. Finally, we make use of a Pohozaev identity type to conclude that System~\eqref{paper1j0} does not admit positive solution when $p=q=6$. 	
	
	
\subsection{Assumptions and main results} In order to establish a variational approach to study System~\eqref{paper1j0}, we introduce some suitable assumptions on the Kirchhoff functions and on the potentials. Throughout the paper, we assume that $a_{1},a_{2}>0$ and $\alpha,\beta\in C^{2}(\mathbb{R}_{+},\mathbb{R}_{+})$. In addition, we suppose that $\alpha$ and $\beta$ satisfy the following assumptions: 

\begin{enumerate}[label=($M_{1}$),ref=$(M_{1})$] 
	\item \label{M_1}
	$\alpha^{\prime}(s)$ and $\beta^{\prime}(t)$ are increasing on $s,t>0$.
\end{enumerate}
\begin{enumerate}[label=($M_{2}$),ref=$(M_{2})$] 
	\item \label{M_3}
	$s\mapsto\displaystyle\frac{\alpha^{\prime}(s)}{s}$ and $t\mapsto\displaystyle\frac{\beta^{\prime}(t)}{t}$ are non-increasing on $s,t>0$.
\end{enumerate}
\begin{enumerate}[label=($M_{3}$),ref=$(M_{3})$] 
	\item \label{M_4}
	$\alpha^{\prime}(s)\leq b_{1}s$, $\beta^{\prime}(t)\leq b_{2}t$ and  
	 \[
	  \frac{1}{2}\alpha^{\prime}(s)s+\frac 12 \beta^{\prime}(t)t\leq \alpha(s)+\beta(t)\leq \alpha^{\prime}(s)s+ \beta^{\prime}(t)t, \quad \mbox{for all} \hspace{0,2cm} s,t \geq 0.
	 \]
\end{enumerate}
\begin{enumerate}[label=($M_{4}$),ref=$(M_{4})$] 
	\item \label{M_5}
	$\alpha^{\prime\prime}(s)s\leq \alpha^{\prime}(s)$ and $\beta^{\prime\prime}(t)t\leq \beta^{\prime}(t)$, for all $s,t\geq0$.
\end{enumerate}

\noindent Due the presence of the potentials $V_{1}$ and $V_{2}$, we have introduced above suitable spaces $E_{1}$ and $E_{2}$. For each $i = 1, 2$, we assume the following hypotheses:
	
	\begin{enumerate}[label=($V_{1}$),ref=$(V_{1})$] 
		\item \label{paper1A1}
		$V_{i  },\lambda \in C(\mathbb{R}^{3},\mathbb{R})$ are $\mathbb{Z}^{3}$-periodic.  
	\end{enumerate}
	
	\begin{enumerate}[label=($V_{2}$),ref=$(V_{2})$] 
		\item \label{paper1A2}
		$V_{i  }(x)\geq0$ for all $x\in\mathbb{R}^{3}$ and 
		$$
		 \inf_{u\in E_{i  }}\left\{\int_{\mathbb{R}^{3}}|\nabla u|^{2}\;\mathrm{d} x+\int_{\mathbb{R}^{3}}V_{i  }(x)u^{2}\;\mathrm{d} x: \int_{\mathbb{R}^{3}}u^{2}\;\mathrm{d} x=1\right\}>0.
		$$			
	\end{enumerate}
	
	\begin{enumerate}[label=($V_{3}$),ref=$(V_{3})$] 
		\item \label{paper1A3}
		$|\lambda (x)|\leq\delta\sqrt{V_{1  }(x)V_{2  }(x)}$, for some $\delta\in\left(0, \min\{a_{1},a_{2} \}\right)$, for all $x\in\mathbb{R}^{3}$.	 
	\end{enumerate}
	
	\begin{enumerate}[label=($V_{3}'$),ref=$(V_{3}')$] 
		\item \label{paper1A9}
		Assumption \ref{paper1A3} holds and $\lambda (x)>0$, for all $x\in\mathbb{R}^{3}$.
	\end{enumerate}

\noindent  In view of \cite[Lemma 2.1]{s}, $E_{i}$ is a Hilbert space continuously embedded into $L^{r}(\mathbb{R}^{3})$ for $r\in[2,6]$ and $i=1,2$. We set the product space $E =E_{1  }\times E_{2  }$. We have that $E $ is a Hilbert space when endowed with the inner product
$$
((u,v),(z,w))_{E }=\int_{\mathbb{R}^{3}}\left(\nabla u\nabla z+V_{1  }(x)uz+\nabla v\nabla w+V_{2  }(x)vw\right)\; \mathrm{d} x,
$$
to which corresponds the induced norm $\|(u,v)\|_{E }^{2}=((u,v),(u,v))_{E }=\|u\|_{E_{1  }}^{2}+\|v\|_{E_{2  }}^{2}$. Associated to System~\eqref{paper1j0} we have the energy functional $I\in C^{1}(E,\mathbb{R})$ defined by
	 \[
	  I (u,v)=\frac{1}{2}\left(a_{1}\|u\|_{E_{1}}^{2}+a_{2}\|v\|_{E_{2}}^{2}\right)+\frac{1}{2}(\alpha(\|u\|_{E_{1}}^{2})+\beta(\|v\|_{E_{2}}^{2}))-\frac{\mu}{p}\|u\|_{p}^{p}-
	  \frac{1}{q}\|v\|_{q}^{q}-\int_{\mathbb{R}^{3}}\lambda (x)uv\; \mathrm{d} x.
	 \]
	By standard arguments it can be checked that critical points of $I$ correspond to weak solutions of \eqref{paper1j0} and conversely. We say that a weak solution $(u_{0},v_{0})\in E$ for System \eqref{paper1j0} is a \textit{ground state solution} (or least energy solution) if $ I(u_{0},v_{0})\leq I(u,v)$ for any other weak solution $(u,v)\in E\backslash\{(0,0)\}$. We say that $(u_{0},v_{0})$ is nonnegative (nonpositive) if $u_{0},v_{0}\geq0$ ($u_{0},v_{0}\leq0$) and positive (negative) if $u_{0},v_{0}>0$ ($u_{0},v_{0}<0$) respectively.
	
	\noindent Now we are able to state our main results.
	
	\begin{theorem}\label{paper1A}
		Assume that \ref{M_1}-\ref{M_5} and \ref{paper1A1}-\ref{paper1A3} hold. If $4<p\leq q<6$, then there exists a nonnegative ground state solution for System~\eqref{paper1j0}, for all $\mu\geq0$. If \ref{paper1A9} holds, then the ground state is positive.
	\end{theorem}
	
	\begin{theorem}\label{paper1B}
		Assume that  \ref{M_1}-\ref{M_5} and  \ref{paper1A1}-\ref{paper1A3} hold. If $4<p<q=6$, then there exists $\mu_{0}>0$ such that System~\eqref{paper1j0} possesses a nonnegative ground state solution $(u_{0},v_{0})\in E $, for all $\mu\geq\mu_{0}$. If \ref{paper1A9} holds, then the ground state is positive.
	\end{theorem}

\begin{theorem}\label{paper1C}
	Let $p=q=6$. In addition, for $i=1,2$ we consider the following assumptions:
	\begin{enumerate}[label=($V_4$),ref=$(V_4)$] 
		\item \label{V4}
		$V_{i}\in C^{1}(\mathbb{R}^{3})$ and $0\leq\langle\nabla V_{i}(x),x\rangle\leq CV_{i}(x)$.
	\end{enumerate}
	
	\begin{enumerate}[label=($V_5$),ref=$(V_5)$] 
		\item \label{V5}
		$\lambda\in C^{1}(\mathbb{R}^{3})$, $|\langle\nabla\lambda(x),x\rangle|\leq C|\lambda(x)|$ and $\langle\nabla\lambda(x),x\rangle\leq0$.
	\end{enumerate}
	Then, System~$\eqref{paper1j0}$ has no positive classical solution for all $\mu\geq0$.
\end{theorem}


\begin{remark}\label{r1}
A typical example of functions verifying assumptions
	\ref{M_1}-\ref{M_5} are given by
	$$
	 \alpha(s)=b_{1}\frac{s^2}{2} \quad \mbox{and} \quad \beta(t)=b_{2}\frac{t^2}{2},	
	$$
	where $b_{1},b_{2}>0$. This is the example that was considered in \cite{kirchhoff} in the scalar case. More generally, 
	$$
	\displaystyle  \alpha(s)=b_{1}\frac{s^2}{2}+\displaystyle \sum_{i=1}^{k}a_{i}s^{\gamma_{i}} \quad \mbox{and} \quad \beta(t)=b_{2}\frac{t^2}{2}+\displaystyle \sum_{i=1}^{k}b_{i}t^{\gamma_{i}},
	$$
	with $a_{i},b_{i}> 0$ and $\gamma_{i}\in (0,1)$ for all $i\in	\{1,2,\ldots, k\}$ verify hypotheses \ref{M_1}-\ref{M_5}. Another example is given by
	 \[
	  \alpha(s)=\int_{0}^{s}\ln(1+r)\,\mathrm{d}r \quad \mbox{and} \quad \beta(t)=\int_{0}^{t}\ln(1+r)\,\mathrm{d}r.
	 \]
	In the present work we introduce a class of Kirchhoff-Schr\"{o}dinger coupled systems by considering different types of functions $\alpha$ and $\beta$.
\end{remark}


\subsection{Notation}
Let us introduce the following notation:	
\begin{itemize}
	\item $C$, $\tilde{C}$, $C_{1}$, $C_{2}$,... denote positive constants (possibly different).
	\item $o_{n}(1)$ denotes a sequence which converges to $0$ as $n\rightarrow\infty$.
	\item The norm in $L^{p}(\mathbb{R}^{3})$ and $L^{\infty}(\mathbb{R}^{3})$, will be denoted respectively by $\|\cdot\|_{p}$ and $\|\cdot\|_{\infty}$.
	\item The norm in $L^{p}(\mathbb{R}^{3})\times L^{p}(\mathbb{R}^{3})$ is given by $\|(u,v)\|_{p}=\left(\|u\|^{p}_{p}+\|v\|^{p}_{p}\right)^{1/p}$.
	\item We write $\int u$ instead of $\int_{\mathbb{R}^{3}}u\;\mathrm{d}x$.
	\item We denote by $S$ the sharp constant of the embedding $D^{1,2}(\mathbb{R}^{3})\hookrightarrow L^{6}(\mathbb{R}^{3})$
	\begin{equation}\label{sharp}
	S\left(\int|u|^{6}\right)^{1/3}\leq \int|\nabla u|^{2},
	\end{equation}
	where $D^{1,2}(\mathbb{R}^{3}):=\{u\in L^{6}(\mathbb{R}^{3}):|\nabla u|\in L^{2}(\mathbb{R}^{3})\}$.	
\end{itemize}


\subsection{Outline}
	In the forthcoming Section we introduce the Nehari manifold associated to System~\eqref{paper1j0}. In Section~\ref{s4} we study the existence of ground states for System~\eqref{paper1j0} in the subcritical case. Section~\ref{s5} is devoted to the critical case. In Section~\ref{s6} we make use of a Pohozaev identity type to prove the nonexistence result.


\section{The Nehari manifold}\label{s3}

The main goal of the paper is to prove the existence of ground state solutions. For this purpose, we use a minimization technique over the Nehari manifold. In order to obtain some properties for the Nehari manifold, we have the following technical lemma:
\begin{lemma}\label{lambda}
	If \ref{paper1A3} holds, then we have 
	\begin{equation}\label{gjj1}
	 a_{1}\|u\|_{E_{1}}^{2}+a_{2}\|v\|_{E_{2}}^{2} - 2 \displaystyle\int\lambda (x)uv \geq \left(\min\{a_{1},a_{2}\}-\delta\right)\|(u,v)\|^{2}_{E }.
	\end{equation} 
\end{lemma}
\begin{proof} Note that
$$
2\sqrt{V_{1  }(x)V_{2  }(x)}|u||v|\leq V_{1  }(x)u^{2} + V_{2  }(x) v^{2}.
$$
Thus, using \ref{paper1A3} we can deduce that
$$
-2 \displaystyle\int \lambda (x)uv \geq - \delta \displaystyle\int\bigl(V_{1  }(x)u^{2} + V_{2  }(x) v^{2}\bigl) \geq - \delta\|(u,v)\|^{2}_{E },
$$
which easily implies \eqref{gjj1}.
\end{proof}

The Nehari manifold associated to System~\eqref{paper1j0} is given by
$$
\mathcal{N} =\left\{(u,v)\in E \backslash\{(0,0)\}: I ^{\prime}(u,v)(u,v)=0\right\}.
$$
Notice that if $(u,v)\in\mathcal{N} $, then
\begin{align}\label{paper1j8}
a_{1}\|u\|_{E_{1}}^{2}+a_{2}\|v\|_{E_{2}}^{2}+\alpha^{\prime}(\|u\|_{E_{1}}^{2})\|u\|_{E_{1}}^{2}+\beta^{\prime}(\|v\|_{E_{2}}^{2})\|v\|_{E_{2}}^{2}-2\int\lambda (x)uv=\mu\|u\|_{p}^{p}+\|v\|_{q}^{q}.
\end{align}

\begin{lemma}\label{paper1nehari}
	Suppose that \ref{paper1A2}, \ref{paper1A3}, \ref{M_4} and \ref{M_5} hold. Then, there exists $\alpha>0$ such that
	\begin{equation}\label{paper1j26}
	\|(u,v)\|_{E }\geq\alpha, \quad \mbox{for all} \hspace{0,2cm} \ (u,v)\in\mathcal{N} .
	\end{equation}
	Moreover, $\mathcal{N} $ is a $C^{1}$-manifold.
\end{lemma}
\begin{proof}
	If $(u,v)\in\mathcal{N} $, then using Lemma~\ref{lambda}, \eqref{paper1j8} and Sobolev embedding we deduce that
	\begin{eqnarray*}
		\left(\min\{a_{1},a_{2}\}-\delta\right)\|(u,v)\|_{E }^{2} & \leq & a_{1}\|u\|_{E_{1}}^{2}+a_{2}\|v\|_{E_{2}}^{2}-2\int\lambda (x)uv\\
		                                            & \leq & \mu\|u\|_{p}^{p}+\|v\|_{q}^{q}\\
		                                            &\leq & C\left(\|(u,v)\|_{E }^{p}+\|(u,v)\|_{E }^{q}\right).
	\end{eqnarray*}
	Hence, we have that
	$$
	0<\frac{\min\{a_{1},a_{2}\}-\delta}{C}\leq\|(u,v)\|_{E }^{p-2}+\|(u,v)\|_{E }^{q-2},
	$$
	which implies \eqref{paper1j26}. In order to prove that $\mathcal{N}$ is a $C^{1}$-manifold, let $ J :E \backslash\{(0,0)\}\rightarrow\mathbb{R}$ be the $C^{1}$-functional given by $J (u,v)=I ^{\prime}(u,v)(u,v)$. Notice that $\mathcal{N} = J ^{-1}(0)$ and 
	 \begin{align*}
	  J ^{\prime}(u,v)(u,v)=2\left(a_{1}\|u\|_{E_{1}}^{2}+a_{2}\|v\|_{E_{2}}^{2}-2\int\lambda (x)uv\right)+2\alpha^{\prime\prime}(\|u\|_{E_{1}}^{2})\|u\|_{E_{1}}^{2}+2\alpha^{\prime}(\|u\|_{E_{1}}^{2})\|u\|_{E_{1}}^{2}+\\+2\beta^{\prime\prime}(\|v\|_{E_{2}}^{2})\|v\|_{E_{2}}^{2}+2\beta^{\prime}(\|v\|_{E_{2}}^{2})\|v\|_{E_{2}}^{2}-\mu p\|u\|_{p}^{p}-q\|v\|_{q}^{q},
	 \end{align*}
	which together with assumption \ref{M_5} implies that
	 \begin{align*}
	 J ^{\prime}(u,v)(u,v)\leq 2\left(a_{1}\|u\|_{E_{1}}^{2}+a_{2}\|v\|_{E_{2}}^{2}-2\int\lambda (x)uv\right)+4\alpha^{\prime}(\|u\|_{E_{1}}^{2})\|u\|_{E_{1}}^{2}+\\+4\beta^{\prime}(\|v\|_{E_{2}}^{2})\|v\|_{E_{2}}^{2}-\mu p\|u\|_{p}^{p}-q\|v\|_{q}^{q}.
	 \end{align*}
	Since $(u,v)\in\mathcal{N} $, we can conclude that
	 \begin{eqnarray*}
	 J ^{\prime}(u,v)(u,v) & \leq & -2\left(a_{1}\|u\|_{E_{1}}^{2}+a_{2}\|v\|_{E_{2}}^{2}-2\int\lambda (x)uv\right)+\mu(4-p)\|u\|_{p}^{p}+(4-q)\|v\|_{q}^{q}\\
	              & \leq & -2\left(\min\{a_{1},a_{2}\}-\delta\right)\|(u,v)\|^{2}_{E }+\mu(4-p)\|u\|_{p}^{p}+(4-q)\|v\|_{q}^{q}<0,
	 \end{eqnarray*}
	where we have used Lemma \ref{lambda} and the fact that $4<p\leq q$. Therefore, $0$ is a regular value of $ J $ and $\mathcal{N} $ is a $C^{1}$-manifold.
\end{proof}

\begin{remark}
	If $(u_{0},v_{0})\in\mathcal{N} $ is a critical point of the constrained functional $I\mid_{\mathcal{N}}$, then $I'(u_{0},v_{0})=0$. In fact, notice that
	$
	I'(u_{0},v_{0})=\eta J'(u_{0},v_{0}),
	$
	where $\eta\in\mathbb{R}$ is the corresponding Lagrange multiplier. Taking the scalar product with $(u_{0},v_{0})$ we conclude that $\eta=0$.
\end{remark}

\begin{lemma}\label{paper1p1}
	Assume \ref{paper1A2}, \ref{paper1A3}, \ref{M_3} and \ref{M_4} hold. Then, for any $(u,v)\in E \backslash\{(0,0)\}$ there exists a unique $t_{0}>0$, depending only on $(u,v)$, such that 
	$$
	(t_{0}u,t_{0}v)\in\mathcal{N}  \quad \mbox{and} \quad I (t_{0}u,t_{0}v)=\max_{t\geq0} I (tu,tv).
	$$
\end{lemma}
\begin{proof}
	Let $(u,v)\in E \backslash\{(0,0)\}$ be fixed and consider the fiber map $g:[0,\infty)\rightarrow\mathbb{R}$ defined by $g(t)= I (tu,tv)$. Notice that
	$
	\langle I ^{\prime}(tu,tv)(tu,tv)\rangle=tg'(t).
	$
	Therefore, $t_{0}$ is a positive critical point of $g$ if and only if $(t_{0}u,t_{0}v)\in\mathcal{N} $. 
	Note that
	$$
	 g(t)= \frac{t^{2}}{2}\left(a_{1}\|u\|_{E_{1}}^{2}+a_{2}\|v\|_{E_{2}}^{2}-2\int\lambda (x)uv\right)+\frac{1}{2}(\alpha(\|tu\|_{E_{1}}^{2})+\beta(\|tv\|_{E_{2}}^{2}))-\mu\frac{t^{p}}{p}\|u\|_{p}^{p}-
	\frac{t^{q}}{q}\|v\|_{q}^{q}.
	$$
	By using Lemma \ref{lambda} and Sobolev embeddings, we have that
	$$
	g(t) \geq (\min\{a_{1},a_{2}\}-\delta)\frac{t^{2}}{2}\|(u,v)\|_{E }^{2}-C_{1}\mu\frac{t^{p}}{p}\|(u,v)\|_{E }^{p}-C_{2}\frac{t^{q}}{q}\|(u,v)\|_{E }^{q}\>>0,
	$$
	provided $t>0$ is sufficiently small. On the other hand, by using \ref{M_4} we deduce that
	 \begin{align*}
	  g(t)\leq \frac{t^{2}}{2}\left(a_{1}\|u\|_{E_{1}}^{2}+a_{2}\|v\|_{E_{2}}^{2}-2\int\lambda (x)uv\right)+b_{1}\frac{t^{4}}{2}\|u\|_{E_{1  }}^{2}+b_{2}\frac{t^{4}}{2}\|v\|_{E_{2  }}^{2}-\mu\frac{t^{p}}{p}\|u\|_{p}^{p}-\frac{t^{q}}{q}\|v\|_{q}^{q}.
	 \end{align*}
	Since $4<p\leq q$ we conclude that $g(t)<0$ for $t>0$ sufficiently large. Thus $g$ has maximum points in $(0,\infty)$. It remains to prove that the critical point is unique. In fact, notice that if $g'(\bar{t})=0$ then
	\begin{align*}	
	 \frac{1}{\bar{t}^{2}}\left(a_{1}\|u\|_{E_{1}}^{2}+a_{2}\|v\|_{E_{2}}^{2}-2\int\lambda (x)uv\right)+\frac{\alpha^{\prime}(\|\bar{t}u\|_{E_{1}}^{2})}{\|\bar{t}u\|_{E_{1}}^{2}}\|u\|_{E_{1}}^{4}+\frac{\beta^{\prime}(\|\bar{t}v\|_{E_{2}}^{2})}{\|\bar{t}v\|_{E_{2}}^{2}}\|v\|_{E_{2}}^{4}=\mu \bar{t}^{p-4}\|u\|_{p}^{p}+\bar{t}^{q-4}\|v\|_{q}^{q}.
	\end{align*}
	It follows from \ref{M_3} that the left-hand side is decreasing on $\bar{t}>0$. Since the right-hand side is increasing on $\bar{t}>0$, the maximum point of $g$ is unique.
\end{proof}

\begin{remark}
	Let us define the following energy levels associated to System~\eqref{paper1j0}:
	$$
	c_{\mathcal{N} }=\inf_{(u,v)\in\mathcal{N} } I (u,v), \quad c^{*}=\displaystyle\inf_{(u,v)\in E \backslash\{(0,0)\}}\max_{t\geq0} I (tu,tv), \quad  c=\displaystyle\inf_{\gamma\in\Gamma}\max_{t\in[0,1]} I (\gamma(t)),
	$$
	where
	$
	\Gamma=\left\{\gamma\in C([0,1],E ):\gamma(0)=(0,0),  I (\gamma(1))<0\right\}.
	$ By a similar argument used in \cite[Theorem 4.2]{will} we can deduce that
	$
	0<c_{\mathcal{N} }=c^{*}=c.
	$
\end{remark}


\section{The subcritical case}\label{s4}

In this section we are concerned to prove existence of ground states for the subcritical System~\eqref{paper1j0}. For this purpose, we follow some ideas from \cite[Theorem~2.5]{alvesg}. Let $(u_{n},v_{n})\subset\mathcal{N} $ be a minimizing sequence for $c_{\mathcal{N} }$, that is
\begin{equation}\label{paper1j6}
I (u_{n},v_{n})\rightarrow c_{\mathcal{N} } \quad \mbox{and} \quad  I ^{\prime}(u_{n},v_{n})\rightarrow0.
\end{equation}

\begin{proposition}
	If \ref{paper1A2}, \ref{paper1A3} and \ref{M_4} hold, then the minimizing sequence $(u_{n},v_{n})$ is bounded in $E $.
\end{proposition}
\begin{proof}
	In fact, recalling that we are assuming $4<p\leq q$, it follows from \eqref{paper1j8} that
	 \begin{align*}
	  I (u_{n},v_{n}) \geq \left(\frac{1}{2}-\frac{1}{p}\right)\left(a_{1}\|u\|_{E_{1}}^{2}+a_{2}\|v\|_{E_{2}}^{2}-2\int\lambda (x)uv\right)+\left(\frac{1}{p}-\frac{1}{q}\right)\|v_{n}\|_{q}^{q}+\\+\frac{1}{2}(\alpha(\|u_{n}\|_{E_{1}}^{2})+\beta(\|v_{n}\|_{E_{2}}^{2}))+\frac{1}{4}(\alpha^{\prime}(\|u_{n}\|_{E_{1}}^{2})\|u_{n}\|_{E_{1}}^{2}+\beta^{\prime}(\|v_{n}\|_{E_{2}}^{2})\|v_{n}\|_{E_{2}}^{2}),
	 \end{align*}
	which together with assumption \ref{M_4} and Lemma~\ref{lambda} implies that
	 $$
	  I (u_{n},v_{n}) \geq \left(\frac{1}{2}-\frac{1}{p}\right)\left(\min\{a_{1},a_{2}\}-\delta\right)\|(u_{n},v_{n})\|^{2}_{E }.
	 $$
	Since $ I (u_{n},v_{n})$ is bounded, we conclude that $(u_{n},v_{n})$ is bounded in $E $.	
\end{proof}

By the preceding Proposition we may assume that, up to a subsequence, we have
	\begin{equation}\label{conv}
	 \left\{
	  \begin{array}{ll}
		(u_{n},v_n)\rightharpoonup (u_0,v_0), \mbox{ weakly in } E ; & \\
		\|u_{n}\|_{E_{1  }}\rightarrow\varrho_{0} \mbox{ and } \|v_n\|_{E_{2  }} \rightarrow \varrho_1, \mbox{ strongly in } \mathbb{R}; & \\			
		(u_{n},v_n) \rightarrow (u_0,v_0), \mbox{ strongly in } L^{t}_{loc}(\mathbb{R}^{3})\times L^{t}_{loc}(\mathbb{R}^{3}) \mbox{ for all } t \in (2,6); &\\
		(u_{n}(x),v_n(x))\rightarrow (u_0(x),v_0(x)), \mbox{ almost everywhere in } \mathbb{R}^{3}. &
	  \end{array}
	 \right.
	\end{equation}
Without loss of generality, we can assume that $(u_0,v_0)\neq (0,0)$. In fact, by using a standard argument we can use the result due to Lions \cite[Lemma 1.21]{will} (see also \cite{Lio}) to prove that there exist $\eta>0$ and $(y_{n})\subset \mathbb{R}^{3}$ such that
	\begin{equation}\label{paper1gjj1}
	 \liminf_{n\to\infty} \int_{B_R(y_n)} (u_n^2 +v_n^{2})\geq \eta > 0.
	\end{equation}
A direct computation shows that we can assume $(y_{n}) \subset \mathbb{Z}^{3}$. Let us define the shift sequence
	$$
	 (\tilde{u}_{n}(x),\tilde{v}_{n}(x))=(u_{n}(x+y_{n}),v_{n}(x+y_{n})).
	$$
In view of assumption \ref{paper1A1} one can see that $(\tilde{u}_{n},\tilde{v}_{n})$ is a Palais-Smale sequence of $I$ at level $c_{\mathcal{N} }$, that is, also satisfies \eqref{paper1j6}. Moreover, we have that $(\tilde{u}_{n},\tilde{v}_{n})$ is also bounded in $E $ and its weak limit denoted by $(\tilde{u}_{0},\tilde{v}_{0})$ is nontrivial, because \eqref{paper1gjj1} and the local convergence imply that
	$$
	 \int_{B_R(0)} (\tilde{u}_{0}^{2} + \tilde{v}_{0}^{2}) \geq \eta > 0.
	$$ 
Therefore, we can assume $(u_{0},v_{0})\neq(0,0)$.
	
\begin{proposition}\label{p2}
	Suppose that \ref{paper1A2}, \ref{paper1A3}, \ref{M_1}--\ref{M_5} hold. Then, the weak limit $(u_{0},v_{0})$ is a critical point of $I$.
\end{proposition}
\begin{proof}
	By using \eqref{conv} and the fact that $\alpha^{\prime}$ and $\beta^{\prime}$ are continuous, we have the convergences
	 $$
	  \alpha^{\prime}(\|u_{n}\|_{E_{1}}^{2})\rightarrow\alpha^{\prime}(\varrho_{0}^{2}) \quad \mbox{and} \quad \beta^{\prime}(\|v_{n}\|_{E_{2}}^{2})\rightarrow\beta^{\prime}(\varrho_{1}^{2}).
	 $$
	Since $I'(u_n,v_n)=o_{n}(1)$, we conclude that $(u_0,v_0)$ is a nontrivial solution of the following system
	 \begin{equation}\label{p1}
	  \left\{
	   \begin{array}{lr}
	    \left(a_{1}+\alpha^{\prime}(\varrho_{0}^{2})\right)(-\Delta u_{0}+V_{1  }(x)u_{0})= \mu|u_0|^{p-2}u_0 + \lambda  (x)v_0,  & x\in\mathbb{R}^{3},\\
	     \left(a_{2}+\beta^{\prime}(\varrho_{1}^{2})\right)(-\Delta v_{0}+V_{2  }(x)v_{0})= |v_0|^{q-2}v_0 +\lambda  (x)u_0, 
	       & x\in\mathbb{R}^{3}.
	    \end{array}
	   \right. 
	 \end{equation}
	To conclude our proof, it suffices to prove that $\alpha^{\prime}(\|u_{0}\|_{E_{1}}^{2})=\alpha^{\prime}(\varrho_{0}^{2})$ and $\beta^{\prime}(\|v_{0}\|_{E_{2}}^{2})=\beta^{\prime}(\varrho_{1}^{2})$. By the lower semicontinuity of the norm we have
	$$
	\displaystyle\liminf_{n\rightarrow \infty}\|u_{n}\|_{E_{1  }} \geq \|u_0\|_{E_{1  }}.
	$$
   Consequently, given $\epsilon>0$, there exists $n_{0}\in\mathbb{N}$ such that
	$$
	\|u_{n}\|_{E_{1  }} \geq\|u_0\|_{E_{1  }}- \epsilon, \quad \mbox{for all} \hspace{0,2cm} n \geq n_{0}.
	$$
   Arguing by the same argument we get
	$$
	 \|v_{n}\|_{E_{2  }} \geq\|v_0\|_{E_{2  }}- \epsilon, \quad \mbox{for all} \hspace{0,2cm} n \geq n_{0}.
	$$
   Thus, for $n\geq n_{0}$, it follows from \ref{M_1} that
	$$
	 \alpha^{\prime}(\|u_{n}\|_{E_{1}}^{2})\geq \alpha^{\prime}(\|u_0\|^{2}_{E_{1  }}- \epsilon) \quad \mbox{and} \quad \beta^{\prime}(\|v_{n}\|_{E_{2}}^{2})\geq \beta^{\prime}(\|v_0\|^{2}_{E_{2  }}- \epsilon)
	$$
   Letting $n \to \infty$, and after $\epsilon \to 0$, we obtain
	$$
	 \alpha^{\prime}(\varrho_{0}^{2})\geq \alpha^{\prime}(\|u_0\|^{2}_{E_{1  }}) \quad \mbox{and} \quad \beta^{\prime}(\varrho_{1}^{2})\geq \beta^{\prime}(\|v_0\|^{2}_{E_{2  }})
	$$
 Let us suppose by contradiction that at least one of the preceding estimates is strictly. Without loss of generality, let us assume that $\alpha^{\prime}(\varrho_{0}^{2})> \alpha^{\prime}(\|u_0\|^{2}_{E_{1  }})$. Let $\mathcal{H}:(0,+\infty)\rightarrow\mathbb{R}$ be defined by
 \[
 \mathcal{H}(t):=\frac{1}{2}(\alpha(\|tu_{0}\|_{E_{1}}^{2})+\beta(\|tv_0\|_{E_{2}}^{2}))-\frac{1}{4}(\alpha^{\prime}(\|tu_{0}\|_{E_{1}}^{2})\|tu_{0}\|_{E_{1}}^{2}+\beta^{\prime}(\|tv_{0}\|_{E_{2}}^{2})\|tv_{0}\|_{E_{2}}^{2}).
 \]
 By using \ref{M_5} one can see that $\mathcal{H}(t)$ is increasing on $t>0$. Since $(u_{0},v_{0})$ is a weak solution for the problem \eqref{p1} we can deduce that $I'(u_0,v_0)(u_0,v_0)<0$. Moreover, we have that $I'(t_{0}u_0,t_{0}v_0)(t_{0}u_0,t_{0}v_0)>0$ for some $t\in(0,1)$. Therefore, there exists $\overline{t} \in (0,1)$ such that $\overline{t}(u_0,v_0) \in \mathcal{N} $. Thus, it follows that
 	\begin{eqnarray*}
		c_{\mathcal{N} } &\leq & I(\overline{t}u_0,\overline{t}v_0)- \displaystyle\frac{1}{4} I'(\overline{t}u_0,\overline{t}v_0)(\overline{t}u_0,\overline{t}v_0)\\
		& = & \frac{\overline{t}^{2}}{4}\left(a_{1}\|u_{0}\|_{E_{1}}^{2}+a_{2}\|v_{0}\|_{E_{2}}^{2}-2\int\lambda (x)u_{0}v_{0}\right)+\mathcal{H}(\overline{t})+\overline{t}^{p}\left(\frac{1}{4}-\frac{1}{p}\right)\mu\|u_{0}\|_{p}^{p}+\overline{t}^{q}\left(\frac{1}{4}-\frac{1}{q}\right)\|v_{0}\|_{q}^{q}.
	\end{eqnarray*}
	Since the right-hand side is strictly increasing on $\overline{t}>0$ one has
	\[
	 c_{\mathcal{N} }< \frac{1}{4}\left(a_{1}\|u_{0}\|_{E_{1}}^{2}+a_{2}\|v_{0}\|_{E_{2}}^{2}-2\int\lambda (x)u_{0}v_{0}\right)+\mathcal{H}(1)+\left(\frac{1}{4}-\frac{1}{p}\right)\mu\|u_{0}\|_{p}^{p}+\left(\frac{1}{4}-\frac{1}{q}\right)\|v_{0}\|_{q}^{q}.
	\]
	By using Fatou's Lemma we conclude that
	$$
	c_{\mathcal{N} }<\liminf_{n\rightarrow\infty}\left\{I(u_n,v_n)- \displaystyle\frac{1}{4} I'(u_n,v_n)(u_n,v_n)\right\}=c_{\mathcal{N} },
	$$
	which is not possible. Therefore, $I'(u_{0},v_{0})=0$.
\end{proof}

\begin{proof}[Proof of Theorem~\ref{paper1A} completed] In the preceding Proposition we have obtained a nontrivial weak solution $(u_{0},v_{0})\in E $ for System~\eqref{paper1j0}. Thus, $(u_{0},v_{0})\in\mathcal{N} $ and consequently $c_{\mathcal{N} }\leq I (u_{0},v_{0})$. On the other hand, by using \ref{M_4}, Lemma \ref{lambda}, the semicontinuity of norm and Fatou's Lemma, we can deduce that
	\begin{eqnarray*}
		c_{\mathcal{N} }+o_{n}(1) & = &  I (u_{n},v_{n})-\frac{1}{4}  I '(u_{n},v_{n})(u_{n},v_{n})\\
		& \geq  &  I (u_{0},v_{0})-\frac{1}{4}  I '(u_{0},v_{0})(u_{0},v_{0})+o_{n}(1)\\
		& = &  I (u_{0},v_{0})+o_{n}(1),
	\end{eqnarray*}  
which implies that $c_{\mathcal{N} }\geq I (u_{0},v_{0})$. Therefore, $I (u_{0},v_{0})=c_{\mathcal{N} }$. In order to get a nonnegative ground state solution, in view of Lemma \ref{paper1p1}, there exists $t_{0}>0$ such that $(t_{0}|u_{0}|,t_{0}|v_{0}|)\in\mathcal{N} $. Thus, we can deduce that
	$$
	I(t_{0}|u_{0}|,t_{0}|v_{0}|)\leq I(t_{0}u_{0},t_{0}v_{0})\leq I(u_{0},v_{0})=c_{\mathcal{N} },
	$$
	which implies that $(t_{0}|u_{0}|,t_{0}|v_{0}|)$ is also a minimizer of $ I$ on $\mathcal{N} $. Therefore, $(t_{0}|u_{0}|,t_{0}|v_{0}|)$ is a nonnegative ground state solution for System~\eqref{paper1j0}. Now, let us suppose that \ref{paper1A9} holds and let us denote $(\bar{u}_{0},\bar{v}_{0})=(t_{0}|u_{0}|,t_{0}|v_{0}|)$. Since $(\bar{u}_{0},\bar{v}_{0})\neq(0,0)$ we may assume without loss of generality that $\bar{u}_{0}\neq0$. We claim that $\bar{v}_{0}\neq0$. In fact, arguing by contradiction we suppose that $\bar{v}_{0}=0$. Thus, since $(\bar{u}_{0},\bar{v}_{0})$ is a critical point of $I$, we deduce that
	$$
	0=\langle I'(\bar{u}_{0},\bar{v}_{0}),(0,\phi)\rangle=-\int\lambda(x)\bar{u}_{0}\phi, \quad \mbox{for all} \hspace{0,2cm} \phi\in C^{\infty}_{0}(\mathbb{R}^{3}).
	$$
	Since $\lambda(x)$ is positive, we have that $\bar{u}_{0}=0$ which is a contradiction. Therefore, $\bar{v}_{0}\neq0$. By using strong maximum principle (see \cite{trud}) in each equation of \eqref{paper1j0}, we conclude that $(\bar{u}_{0},\bar{v}_{0})$ is positive, which finishes the proof of Theorem~\ref{paper1A}.
\end{proof}


\section{The critical case}\label{s5}

In the preceding Section we have obtained a positive ground state solution to the subcritical System~\eqref{paper1j0}, for any parameter $\mu>0$. In this Section, we are concerned with the critical case, that is, when $2<p<q=6$. Analogously to the subcritical case, we have a minimizing sequence $(u_{n},v_{n})\subset\mathcal{N} $ satisfying \eqref{paper1j6}. Moreover, the sequence is bounded and $(u_{n},v_{n})\rightharpoonup (u_{0},v_{0})$ weakly in $E $. We may assume that, up to a subsequence, $\|v_{n}\|_{6}^{6}\rightarrow A_{\mu}\in[0,+\infty)$. If $A_{\mu}=0$, then the proof follows by the same arguments from the subcritical case, using \cite[Lemma 1.21]{will} to obtain \eqref{paper1gjj1} and a nontrivial critical point for the energy functional $I $. Let us assume $A_{\mu}>0$. In this case, we do not obtain \eqref{paper1gjj1} directly, since $q=2^{*}$ and we are not able to use strong convergence in the argument. In order to overcome this difficulty, we choose $\mu>0$ large enough such that the level $c_{\mathcal{N} }$ be in the suitable range where we can recover the compactness.

\begin{proposition}\label{paper1mu}
	Suppose that \ref{paper1A2}, \ref{paper1A3}, \ref{M_3} and \ref{M_4} hold. There exists $\mu_{0}>0$ such that 
	 \begin{equation}\label{level} 
	  c_{\mathcal{N} }<\left(\frac{1}{4}-\frac{1}{p}\right)\left[\left(\min\{a_{1},a_{2}\}-\delta\right)S\right]^{3/2}, \quad \mbox{for all} \hspace{0,2cm} \mu\geq\mu_{0}.
	 \end{equation}  
\end{proposition}
\begin{proof}
	Let $(u,v)\in E$ be nonnegative and $u,v\not\equiv0$. We define $u_{\mu}=\mu u$ and $v_{\mu}=\mu v$. It follows from Lemma~\ref{paper1p1} that for any $\mu>0$, there exists a unique $t_{\mu}>0$ such that $(t_{\mu}u_{\mu},t_{\mu}v_{\mu})\in\mathcal{N} $. Hence, we have
	\begin{align}\label{paper1j51}    
	 (t_{\mu}\mu)^{2}\left(a_{1}\|u\|_{E_{1}}^{2}+a_{2}\|v\|_{E_{2}}^{2}-2\int\lambda (x)uv\right)+(t_{\mu}\mu)^{2}\alpha^{\prime}(\|t_{\mu}u_{\mu}\|_{E_{1}}^{2})\|u\|_{E_{1}}^{2}+\nonumber\\+(t_{\mu}\mu)^{2}\beta^{\prime}(\|t_{\mu}v_{\mu}\|_{E_{2}}^{2})\|v\|_{E_{2}}^{2} =\mu (t_{\mu}\mu)^{p}\|u\|_{p}^{p}+(t_{\mu}\mu)^{6}\|v\|_{6}^{6},
	\end{align}
	which together with \ref{M_4} implies that
	 $$
	  (t_{\mu}\mu)^{2}\left(a_{1}\|u\|_{E_{1}}^{2}+a_{2}\|v\|_{E_{2}}^{2}-2\int\lambda (x)uv\right)+(t_{\mu}\mu)^{4}(b_{1}\|u\|_{E_{1  }}^{4}+b_{2}\|v\|_{E_{2  }}^{4})\geq (t_{\mu}\mu)^{6}\|v\|_{6}^{6}.
	 $$
	Therefore, $(t_{\mu}\mu)_{\mu}$ is a real bounded family. Hence, up to a subsequence, $t_{\mu}\mu\rightarrow t_{0}\geq0$, as $\mu\rightarrow+\infty$. We claim that $t_{0}=0$. Indeed, let us suppose by contradiction that $t_{0}>0$. Thus, it follows from \eqref{paper1j51} that
	$$
	 (t_{\mu}\mu)^{2}\left(a_{1}\|u\|_{E_{1}}^{2}+a_{2}\|v\|_{E_{2}}^{2}-2\int\lambda (x)uv\right)+(t_{\mu}\mu)^{4}(b_{1}\|u\|_{E_{1  }}^{4}+b_{2}\|v\|_{E_{2  }}^{4})\geq\mu (t_{\mu}\mu)^{p}\|u\|_{p}^{p},
	$$
	which is not possible since the right-hand side goes to infinity as $\mu\rightarrow+\infty$. Therefore, $t_{\mu}\mu\rightarrow0$ as $\mu\rightarrow+\infty$. By using \ref{M_4} and the fact that $(t_{\mu}u_{\mu},t_{\mu}v_{\mu})\in\mathcal{N} $, we have
	 $$
	  c_{\mathcal{N}} \leq (t_{\mu}\mu)^{2}\left(a_{1}\|u\|_{E_{1}}^{2}+a_{2}\|v\|_{E_{2}}^{2}-2\int\lambda (x)uv\right)+(t_{\mu}\mu)^{4}(b_{1}\|u\|_{E_{1  }}^{4}+b_{2}\|v\|_{E_{2  }}^{4}).
	 $$
	Therefore, there exists $\mu_{0}>0$ such that \eqref{level} holds, for all $\mu\geq\mu_{0}$.
\end{proof}

\begin{corollary}
	Assume that \ref{paper1A2}, \ref{paper1A3}, \ref{M_3} and \ref{M_4} hold. Let $(u_{n},v_{n})$ be the minimizing sequence satisfying \eqref{paper1j6}. If $\mu\geq\mu_{0}$, then there exists $(y_{n})\subset\mathbb{R}^{3}$ and constants $R,\eta>0$ such that
	\begin{equation}\label{paper1j500}
	\liminf_{n\rightarrow\infty}\int_{B_{R}(y_{n})}(u_{n}^{2}+v_{n}^{2})\geq\eta>0.
	\end{equation}
\end{corollary}
\begin{proof}
Arguing by contradiction, we suppose that \eqref{paper1j500} does not hold. Hence, for any $R>0$ we have
$$
\lim_{n\rightarrow\infty}\sup_{y\in\mathbb{R}^{3}}\int_{B_{R}(y)}(u_{n}^{2}+v_{n}^{2})=0.
$$
In light of \cite[Lemma 1.21]{will} (see also \cite{Lio}) we have that $u_{n}\rightarrow0$ and $v_{n}\rightarrow0$ strongly in $L^{p}(\mathbb{R}^{3})$, for $2<p<2^{*}$. Thus, it follows from \eqref{paper1j8} that
 \begin{align}\label{gjj2}
  a_{1}\|u_{n}\|_{E_{1}}^{2}+a_{2}\|v_{n}\|_{E_{2}}^{2}-2\int\lambda (x)u_{n}v_{n}+\alpha^{\prime}(\|u_{n}\|_{E_{1}}^{2})\|u_{n}\|_{E_{1}}^{2}+\beta^{\prime}(\|v_{n}\|_{E_{2}}^{2})\|v_{n}\|_{E_{2}}^{2}=A_{\mu}+o_{n}(1).
 \end{align}
By using \eqref{sharp}, Lemma~\ref{lambda} and \eqref{gjj2} we deduce that
 \begin{eqnarray*}
  \left(\min\{a_{1},a_{2}\}-\delta\right)SA_{\mu}^{1/3}+o_{n}(1) & = & \left(\min\{a_{1},a_{2}\}-\delta\right)S\|v_{n}\|_{6}^{2}\\
                                                                          & \leq & \left(\min\{a_{1},a_{2}\}-\delta\right)\|(u_{n},v_{n})\|^{2}_{E}\\
                                                                          & \leq & a_{1}\|u_{n}\|_{E_{1}}^{2}+a_{2}\|v_{n}\|_{E_{2}}^{2} - 2 \displaystyle\int\lambda (x)u_{n}v_{n}\\
                                                                          & \leq & A_{\mu}+o_{n}(1),
 \end{eqnarray*}
which implies that
 \begin{equation}\label{gjj3}
  A_{\mu}\geq\left[\left(\min\{a_{1},a_{2}\}-\delta\right)S \right]^{3/2}.
 \end{equation}
By using \ref{M_4}, \eqref{paper1j6}, \eqref{gjj2} and \eqref{gjj3} we can deduce that
 \begin{eqnarray*}
  c_{\mathcal{N} } +o_{n}(1) & = & I (u_{n},v_{n})-\frac{1}{4}  I '(u_{n},v_{n})(u_{n},v_{n})\\
                                         & \geq & \left(\frac{1}{4}-\frac{1}{p}\right)A_{\mu}+o_{n}(1)\\ 
                                         & \geq & \left(\frac{1}{4}-\frac{1}{p}\right)\left[\left(\min\{a_{1},a_{2}\}-\delta\right)S\right]^{3/2},
\end{eqnarray*}
which contradicts Proposition~\ref{paper1mu}.
\end{proof} 

\begin{proof}[Proof of Theorem \ref{paper1B} completed]
 Since \eqref{paper1j500} holds, we can introduce the shift sequence $(\tilde{u}_{n}(x),\tilde{v}_{n}(x))=(u_{n}(x+y_{n}),v_{n}(x+y_{n}))$ and we are able to repeat the same arguments used in the proof of Theorem~\ref{paper1A} to finish the proof of Theorem~\ref{paper1B}.
\end{proof}


\section{The nonexistence result}\label{s6}

In this section we are concerned to prove that does not exist positive classical solution for System~\eqref{paper1j0} when $p=q=6$. Let us denote 
\[
f(x,u,v)=-(a_{1}+\alpha^{\prime}(\|u\|_{E_{1}}^{2}))V_{1}(x)u+\mu|u|^{4}u+\lambda(x)v,
\]
\[
g(x,u,v)=-(a_{2}+\beta^{\prime}(\|v\|_{E_{2}}^{2}))V_{2}(x)v+|v|^{4}v+\lambda(x)u.
\]
Thus, we write System~\eqref{paper1j0} in the following form
 \begin{equation}\label{ne1}
  \left\{
  \begin{array}{lr}
  -(a_{1}+\alpha^{\prime}(\|u\|_{E_{1}}^{2}))\Delta u = f(x,u,v), & x\in\mathbb{R}^{3},\\
  -(a_{2}+\beta^{\prime}(\|v\|_{E_{2}}^{2}))\Delta v = g(x,u,v),    & x\in\mathbb{R}^{3}.
  \end{array}
  \right.
 \end{equation}
In order to obtain a nonexistence result we introduce the following Pohozaev type identity:
 \begin{lemma}\label{poho}
 	If $(u,v)\in E$ is a classical solution of System~\eqref{ne1}, then satisfies the following Pohozaev identity
 	 \begin{align*}
 	  (a_{1}+\alpha^{\prime}(\|u\|_{E_{1}}^{2}))\int\left(|\nabla u|^{2}+3V_{1}(x)u^{2}\right)+(a_{2}+\beta^{\prime}(\|v\|_{E_{2}}^{2}))\int\left(|\nabla v|^{2}+3V_{2}(x)v^{2}\right)=\\=2\int\langle\nabla\lambda(x),x\rangle uv-\int\left((a_{1}+\alpha^{\prime}(\|u\|_{E_{1}}^{2}))\langle\nabla V_{1}(x),x\rangle u^{2}+(a_{2}+\beta^{\prime}(\|v\|_{E_{2}}^{2}))\langle\nabla V_{2}(x),x\rangle v^{2}\right)+\\+\int\left(\mu u^{6}+v^{6}+6\lambda(x)uv\right).
 	 \end{align*}
 \end{lemma}
\begin{proof}
	The proof is quite similar to \cite{jr1,will,li} but for the sake of convenience we give a sketch here. Let $(u,v)\in E$ be a classical solution of System~\eqref{ne1}. We introduce the cut-off function $\psi\in C^{\infty}_{0}(\mathbb{R})$ given by $\psi(t)=1$ if $|t|\leq 1$, $\psi(t)=0$ if $|t|\geq 2$ and $|\psi^{\prime}(t)|\leq C$, for some $C>0$. Moreover, we define $\psi_{n}(x)=\psi\left(|x|^{2}/n^{2}\right)$ and we note that $\nabla\psi_{n}(x)=\frac{2}{n^{2}}\psi^{\prime}\left(|x|^{2}/n^{2}\right)x$. Multiplying the first equation in \eqref{ne1} by the factor $\langle\nabla u,x\rangle\psi_{n}$, the second equation by the factor $\langle\nabla v,x\rangle\psi_{n}$, summing and integrating we get
	 \begin{align}\label{paper1pohoz}
	  -\int\left[(a_{1}+\alpha^{\prime}(\|u\|_{E_{1}}^{2}))\Delta u\langle\nabla u,x\rangle+(a_{2}+\beta^{\prime}(\|v\|_{E_{2}}^{2}))\Delta v\langle\nabla v,x\rangle\right]\psi_{n}=\nonumber\\=\int\left[f(x,u,v)\langle\nabla u,x\rangle+g(x,u,v)\langle\nabla v,x\rangle\right]\psi_{n}
	 \end{align}
	The idea is to take the limit as $n\rightarrow+\infty$ in \eqref{paper1pohoz}. Similarly to \cite{jr1,will}, we deduce that
	 \begin{equation*}\label{paper1p7}
	 -(a_{1}+\alpha^{\prime}(\|u\|_{E_{1}}^{2}))\lim_{n\rightarrow\infty}\int\Delta u\langle\nabla u,x\rangle=-(a_{1}+\alpha^{\prime}(\|u\|_{E_{1}}^{2}))\frac{1}{2}\int|\nabla u|^{2},
	 \end{equation*}
	 \begin{equation*}\label{paper1p10}
	 -(a_{2}+\beta^{\prime}(\|v\|_{E_{2}}^{2}))\lim_{n\rightarrow\infty}\int\Delta v\langle\nabla v,x\rangle=-(a_{2}+\beta^{\prime}(\|v\|_{E_{2}}^{2}))\frac{1}{2}\int|\nabla v|^{2}.
	 \end{equation*}
	In order to study the limit in the right-hand side of \eqref{paper1pohoz}, we note that
	$$
	\mbox{div}\left(\psi_{n}F(x,u,v)x\right) =\psi_{n}\langle\nabla F(x,u,v),x\rangle+F(x,u,v)\langle\nabla\psi_{n},x\rangle+3\psi_{n}F(x,u,v),
	$$
	where $F(x,u,v)=-(a_{1}+\alpha^{\prime}(\|u\|_{E_{1}}^{2}))(1/2)V_{1}(x)u^{2}+(\mu/6)u^{6}+\lambda(x)uv$. Notice that
	 \begin{align*}
	  \langle\nabla F(x,u,v),x\rangle=-(a_{1}+\alpha^{\prime}(\|u\|_{E_{1}}^{2}))\frac{1}{2}\langle\nabla V_{1}(x),x\rangle u^{2}+f(x,u,v)\langle\nabla u,x\rangle+\\+\langle\nabla\lambda(x),x\rangle uv+\lambda(x)\langle u\nabla v,x\rangle.
	 \end{align*}
	Thus, we have that
	 	\begin{align*}
	 \int f(x,u,v)\langle\nabla u,x\rangle\psi_{n} = \int\left(\mbox{div}(\psi_{n}F(x,u,v)x)-F(x,u,v)\langle\nabla\psi_{n},x\rangle-3F(x,u,v)\psi_{n}\right)+\\ +\int\left[(a_{1}+\alpha^{\prime}(\|u\|_{E_{1}}^{2}))\frac{1}{2}\langle\nabla V_{1}(x),u\rangle u^{2}-\langle\nabla\lambda(x),x\rangle uv-\lambda(x)\langle u\nabla v,x\rangle\right]\psi_{n}.
	 \end{align*}
	Analogously, denoting $G(x,u,v)=-(a_{2}+\beta^{\prime}(\|v\|_{E_{2}}^{2}))(1/2)V_{2}(x)v^{2}+(1/6)v^{6}+\lambda(x)uv$ we deduce 
	 \begin{align*}
	  \int g(x,u,v)\langle\nabla v,x\rangle\psi_{n}  = \int \left(\mbox{div}(\psi_{n}G(x,u,v)x)-G(x,u,v)\langle\nabla\psi_{n},x\rangle -3G(x,u,v)\psi_{n}\right)+ \\  +\int \left[(a_{2}+\beta^{\prime}(\|v\|_{E_{2}}^{2}))\frac{1}{2}\langle\nabla V_{2}(x),v\rangle v^{2}-\langle\nabla\lambda(x),x\rangle uv-\lambda(x)\langle v\nabla u,x\rangle\right]\psi_{n}.
	 \end{align*}
	By using integration by parts we conclude that
	 \[
	  -\lim_{n\rightarrow\infty}\int \lambda(x)\langle u\nabla v+v\nabla u,x\rangle\psi_{n} =\int\langle\nabla\lambda(x),x\rangle uv+3\int \lambda(x)uv.
	 \]
	Therefore, using integration by parts and Lebesgue dominated convergence theorem we obtain
	\begin{align*}
	\lim_{n\rightarrow\infty}\int \left(f(x,u,v)\langle\nabla u,x\rangle+g(x,u,v)\langle\nabla v,x\rangle\right)\psi_{n}  = -3\int (F(x,u,v)+G(x,u,v))+ \\ +\frac{1}{2}\int \left((a_{1}+\alpha^{\prime}(\|u\|_{E_{1}}^{2}))\langle\nabla V_{1}(x),x\rangle u^{2}+(a_{2}+\beta^{\prime}(\|v\|_{E_{2}}^{2}))\langle\nabla V_{2}(x),x\rangle v^{2}\right) - \\-\int \langle\nabla\lambda(x),x\rangle uv +3\int \lambda(x)uv.
	\end{align*} 
	Replacing $F(x,u,v)$ and $G(x,u,v)$ in the equation above, we get the right-hand side of \eqref{paper1pohoz} which finishes the proof.
\end{proof}

Now, we are able to prove that System~\eqref{ne1} does not admit positive classical solution.

\begin{proof}[Proof of Theorem~\ref{paper1C} completed]
	Let $(u,v)\in E$ be a positive classical solution of System~\eqref{ne1}. By the definition of weak solution we have
	 \begin{equation}\label{n2}
	  (a_{1}+\alpha^{\prime}(\|u\|_{E_{1}}^{2}))\|u\|_{E_{1}}^{2}+(a_{2}+\beta^{\prime}(\|v\|_{E_{2}}^{2}))\|v\|_{E_{2}}^{2}=\mu\int u^{6}+\int v^{6}+2\int\lambda(x)uv.
	 \end{equation}
	Combining Lemma~\ref{poho} and \eqref{n2} we deduce that
	 \begin{align*}
	  \int\left[(a_{1}+\alpha^{\prime}(\|u\|_{E_{1}}^{2}))V_{1}(x)u^{2}+(a_{2}+\beta^{\prime}(\|v\|_{E_{2}}^{2}))V_{2}(x)v^{2}-2\lambda(x)uv\right]=\int\langle\nabla \lambda(x),x\rangle uv-\\
	  -\frac{1}{2}\int\left[(a_{1}+\alpha^{\prime}(\|u\|_{E_{1}}^{2}))\langle\nabla V_{1}(x),x\rangle u^{2}+(a_{2}+\beta^{\prime}(\|v\|_{E_{2}}^{2}))\langle\nabla V_{2}(x),x\rangle v^{2}\right],
	 \end{align*}
	which together with assumptions \ref{V4} and \ref{V5} implies that
	 \begin{equation}\label{gjj11}
	  \int\left[(a_{1}+\alpha^{\prime}(\|u\|_{E_{1}}^{2}))V_{1}(x)u^{2}+(a_{2}+\beta^{\prime}(\|v\|_{E_{2}}^{2}))V_{2}(x)v^{2}-2\lambda(x)uv\right]\leq 0.
	 \end{equation}
	On the other hand, by using assumption \ref{paper1A3} and \eqref{gjj11} we have
	 \begin{eqnarray*}
	  0 & \leq & \frac{1}{\min\{a_{1},a_{2}\}}\int\left(a_{1}V_{1}(x)u^{2}+a_{2}V_{2}(x)v^{2}-2\min\{a_{1},a_{2}\}\sqrt{V_{1}(x)V_{2}(x)}uv\right)\\
	    & \leq & \frac{1}{\min\{a_{1},a_{2}\}}\int\left(a_{1}V_{1}(x)u^{2}+a_{2}V_{2}(x)v^{2}-2\frac{\min\{a_{1},a_{2}\}}{\delta}\lambda(x)uv\right)\\
	    & < & \frac{1}{\min\{a_{1},a_{2}\}}\int\left[(a_{1}+\alpha^{\prime}(\|u\|_{E_{1}}^{2}))V_{1}(x)u^{2}+(a_{2}+\beta^{\prime}(\|v\|_{E_{2}}^{2}))V_{2}(x)v^{2}-2\lambda(x)uv\right]\leq0,
	 \end{eqnarray*}
   which is a contradiction and finishes the proof.
\end{proof}

\section*{Acknowledgement(s)}

The authors would like to express their sincere gratitude to the referee for careful reading the manuscript and valuable comments and suggestions.



\begin{thebibliography}{99}
	
\bibitem{alves}
C.O. Alves, F.J.S.A. Corr\^{e}a, T.F. Ma,
\textit{Positive solutions for a quasilinear elliptic equation of Kirchhoff type,} 
Comput. Math. Appl. {\bf 49}, (2005) 85-–93.
	

\bibitem{alvesg}
C.O. Alves, G.M. Figueiredo,
\textit{Nonlinear perturbations of a periodic Kirchhoff equation in $\mathbb{R}^{N}$,} 
Nonlinear Anal. {\bf 75}, (2012) 2750--2759.






\bibitem{chen}
C. Chen, Y. Kuo, T.F. Wu,
\textit{The Nehari manifold for a Kirchhoff type problem involving sign-changing weight functions}, 
J. Differential Equations {\bf 250}, (2011) 1876--1908.	 	






\bibitem{jr1} 
J.M. do \'{O}, J.C. de Albuquerque,
\textit{Ground states for a linearly coupled system of Schr\"{o}dinger equations on $\mathbb{R}^{N}$}, 
to appear in Asymptotic Analysis.


\bibitem{trud} 
D. Gilbarg, N.S. Trudinger,
\textit{Elliptic Partial Differential Equations of Second Order}, 
2nd ed., Grundlehren der Mathematischen Wissenschaften, vol. 224, Springer-Verlag, Berlin, (1983).


\bibitem{he} 
X.M. He, W.M. Zou,
\textit{Infinitely many positive solutions for Kirchhoff-type problems}, 
Nonlinear Anal., {\bf 70} (2009), 1407--1414.


\bibitem{lu} 
D. L\"{u}, J. Xiao,
\textit{Existence and multiplicity results for a coupled system of Kirchhoff type equations}, 
Electron. J. Qual. Theory Differ. Equ., {\bf 6} (2014), 10pp.


\bibitem{xiao} 
D. L\"{u}, J. Xiao,
\textit{Ground state solutions for a coupled Kirchhoff-type system}, 
Math. Methods Appl. Sci., {\bf 38} (2015), 4931--4948.


\bibitem{kirchhoff}
G. Kirchhoff, \textit{ Mechanik, Teubner,Leipzig, (1883)}.


\bibitem{li} 
Y. Li, F. Li, J. Shi,
\textit{Existence of positive solutions to Kirchhoff type problems with zero mass}, 
J. Math. Anal. Appl., {\bf 410} (2014), 361--374.


\bibitem{Lio} {P.L. Lions}, \textit{ The concentration-compactness principle in the calculus of
variations. The locally compact case. II}, Ann. Inst. H. Poincar\'e
Anal. Non Lin\'eaire \textbf{1} (1984), 223-283.


\bibitem{lions} {J.L. Lions}, \textit{On some questions in boundary value problems of mathematical physics}, North-Holland Math. Stud. \textbf{30} North-Holland, Amsterdam-New York, (1978), 284--346.


\bibitem{s} 
B.~Sirakov,
\textit{Existence and multiplicity of solutions of semi-linear elliptic equations in $\mathbb{R}^{N}$,}
Calc. Var., {\bf 11} (2000), 119-142.



\bibitem{shi} {H. Shi, H. Chen}, 
\textit{Ground state solutions for asymptotically periodic coupled Kirchhoff-type systems with critical growth}, 
Math. Methods Appl. Sci. \textbf{39} (2016), 2193--2201.


\bibitem{tang} {X.H. Tang, S. Chen}, 
\textit{Ground state solutions of Nehari-Pohozaev type for Kirchhoff-type problems with general potentials}, 
Calc. Var. \textbf{56} (2017), 25pp.


\bibitem{will} 
M.~Willem, 
\textit{Minimax Theorems,}
Birkh\"{a}ser, Boston, (1996).



\end{thebibliography}
\end{document}